\documentclass{amsart}

\newtheorem{theorem}{Theorem}[section]           
                
\newtheorem{corollary}{Corollary}[section]
\newtheorem{lem}{Lemma}[section]  

\theoremstyle{definition}

\newtheorem{remark}{Remark}[section]

\newcommand{\C}{{\mathbb{C}}}

\newcommand{\RR}{{\mathcal{R}_n}}  

\begin{document}

\title{Bounds of the Derivative of Some Classes of Rational Functions}

\author{Nuttapong Arunrat}
\address{Department of Mathematics, Faculty of Science, Khon Kaen University, 40002, Thailand}
\email{nutaru36@gmail.com}

\author{Keaitsuda Maneeruk Nakprasit}
\address{Department of Mathematics, Faculty of Science, Khon Kaen University, 40002, Thailand}
\email{kmaneeruk@hotmail.com}





\begin{abstract}
Let $r(z)$ be a rational function with at most $n$ poles, $a_1, a_2, \ldots, a_n,$ where $|a_j| > 1,$ $1\leq j\leq n.$
This paper investigates the estimate of the modulus of the derivative of a rational function $r(z)$ on the unit circle. We establish an upper bound when all zeros of $r(z)$ lie in $|z|\geq k\geq 1$ and a lower bound when all zeros of $r(z)$ lie in $|z|\leq  k \leq 1.$ In particular,  when $k=1$ and $r(z)$ has exactly $n$ zeros, we obtain a generalization of results by A. Aziz and W. M. Shah [Some refinements of Bernstein-type inequalities for rational functions, Glas. Mat., {\bf 32}(52) (1997), 29--37.].
\end{abstract}

\maketitle

\section{Introduction and statement of results}
\indent Let $\mathcal{P}_n$ denote the class of all complex polynomials of degree at most $n$ and let $k$ be a positive real number. We denote $T_k = \{ z : |z|=k\}, D_{k-} = \{ z : |z|< k\},$ and $D_{k+} =\{ z : |z|>k\}.$  For $a_j \in \C$ such that $1 \leq j \leq  n,$ let 
			$$w(z) = \prod_{j=1}^n (z-a_j)$$
and let 
			$$B(z) = \prod_{j=1}^n \left(\dfrac{1-\overline{a_j}z}{z-a_j}\right), \quad \mathcal{R}_n = \mathcal{R}_n(a_1, a_2,\ldots, a_n) := \left\{ \dfrac{p(z)}{w(z)} : p \in \mathcal{P}_n\right\}.$$
The product $B(z)$ is known as a Blaschke product.\\ We can show that $|B(z)|=1$ and $\dfrac{zB'(z)}{B(z)}=|B'(z)|$ for $z \in T_1.$\\
Then $\mathcal{R}_n$ is the set of rational functions with at most $n$ poles $a_1, a_2,\ldots, a_n$ and with finite limit at infinity. For $f$ defined on $T_1$ in the complex plane, we denote $||f|| = \sup_{z \in T_1} |f(z)|$, the \emph{Chebyshev norm} of $f$ on $T_1.$ Throughout this paper, we always assume that all poles $a_1, a_2, \ldots, a_n$ are in $D_{1+}.$

\indent In 1995, X. Li, R. N. Mohapatra and  R. S. Rodriguez \cite{Li} proved Bernstein-type inequalities for rational functions $r(z) \in \mathcal{R}_n$ with prescribed poles in the Chebychev norm on the unit circle. Among other things they proved the following results for rational functions with restricted zeros.

\begin{theorem}\label{thm-li} \cite{Li}
		Let $r \in \mathcal{R}_n$ with all its zeros lie in $T_1 \cup D_{1+}.$ Then for $z \in T_1$, 
				\begin{equation}
							|r'(z)| \leq \dfrac{1}{2} |B'(z)| \cdot ||r||. \label{eq1.1}
				\end{equation}
		Equality holds for $r(z) = a B(z) + b$ with $|a|=|b| =1.$
\end{theorem}

\begin{theorem}\label{thm-li2} \cite{Li}
		Let $r \in \mathcal{R}_n$, where $r$ has exactly $n$ poles at $a_1,a_2,\ldots,a_n$ and all its zeros lie in $T_1 \cup D_{1-}.$ 
Then for $z \in T_1$, 
			\begin{equation}
					|r'(z)| 	\geq 		\left[\dfrac{1}{2}|B'(z)| - \dfrac{1}{2}(n-t)\right]\cdot|r(z)|,   \label{eq1.4}
			\end{equation}
		where $t$ is the number of zeros of $r.$ Equality holds for $r(z) = a B(z) + b$ with $|a|=|b| =1.$
\end{theorem}

\begin{remark} In particular, if $r$ has exactly $n$ zeros in $T_1 \cup D_{1-}$, then the inequality \eqref{eq1.4} yields Bernstein-type inequality, for $z \in T_1$,
		\begin{equation}
						|r'(z)| \geq \dfrac{1}{2}|B'(z)| |r(z)|. \label{eq1.5}
		\end{equation}
\end{remark}

\indent In 1997, inequalities \eqref{eq1.1} and \eqref{eq1.5} were improved by  A. Aziz and W. M. Shah \cite{AzizShah} under the same hypothesis. They obtained the following theorems.

\begin{theorem} \label{thm-azizshah1} 	\cite{AzizShah}
		Let $r \in \mathcal{R}_n$ with all its zeros lie in $T_1 \cup D_{1+}.$ Then for $z \in T_1$, 
				\begin{equation}
						|r'(z)| \leq \dfrac{1}{2} |B'(z)| \left(||r||-m\right), \label{eq1.2}
				\end{equation}
		where $m = \displaystyle \min_{|z|=1} |r(z)|.$ Equality holds for $r(z) = B(z) + he^{i\alpha}$ where $h\geq 1$ and \mbox{$\alpha$ is real.}
\end{theorem}

\begin{theorem}\label{thm-azizshah2}\cite{AzizShah}
		Let $r \in \mathcal{R}_n$ , where $r$ has exactly $n$ zeros and all its zeros lie in $T_1 \cup D_{1-}.$ Then for $z \in T_1$, 
				\begin{equation}
						|r'(z)| \geq \dfrac{1}{2} |B'(z)| \left(|r(z)|+m\right) \label{eq1.6}
				\end{equation}
		where $m = \displaystyle \min_{|z|=1} |r(z)|.$ Equality holds for $r(z) = B(z) + he^{i\alpha}$ where $h\leq 1$ and \mbox{$\alpha$ is real.}
\end{theorem}

\indent In 1999, A. Aziz and B. A. Zarger \cite{AzizZar} considered a class of rational functions $ \mathcal{R}_n$ not vanishing in $D_{k-}$, where $k \geq 1$ and established the following generalization of Theorem \ref{thm-li}.

\begin{theorem}\label{thm-azizzar} \cite{AzizZar}
		Let $r \in \mathcal{R}_n$ with all its zeros lie in $T_k \cup D_{k+},$ where $k \geq 1.$ Then for $z \in T_1$,
				\begin{equation}
						|r'(z)| \leq \dfrac{1}{2} \left[ |B'(z)| - \dfrac{n(k-1)}{k+1} \cdot \dfrac{|r(z)|^2}{||r||^2} \right] \cdot ||r||.\label{eq1.3}
				\end{equation}
		Equality holds for $r(z) = \left(\dfrac{z+k}{z-a}\right)^n$ and \mbox{$B(z) = \left(\dfrac{1-az}{z-a}\right)^n$} evaluated at $z=1$, where $a >1, k\geq 1.$
\end{theorem}

\indent In 2004, A. Aziz and W. M. Shah \cite{AzizShah2} considered a class of rational functions $\RR$ not vanishing in $D_{k+}$, where $k \leq 1$ and they proved the following generalization of Theorem \ref{thm-li2}.

\begin{theorem}\label{thm-azizshahnew}\cite{AzizShah2}
		Let $r \in \mathcal{R}_n$ where $r$ has exactly $n$ poles at $a_1, a_2, \ldots, a_n$ and all its zeros lie in $T_k \cup D_{k-},$ where $k \leq 1.$ Then for $z \in T_1$, 
				\begin{equation}
							|r'(z)| \geq \dfrac{1}{2} \left[ |B'(z)| + \dfrac{2t - n(1+k)}{1+k} \right] \cdot |r(z)|, \label{eq1.7}
				\end{equation}
		where $t$ is the number of zeros of $r.$ Equality holds for $r(z) = \dfrac{(z+k)^t}{(z-a)^n}$ and $B(z) = \left(\dfrac{1-az}{z-a}\right)^n$ evaluated at $z=1$, where $a > 1, k\leq 1.$ 
\end{theorem}

\indent As an immediate consequence of Theorem \ref{thm-azizshahnew}, they obtained the  generalization of inequality \eqref{eq1.5}, where $r$ has exactly $n$ zeros in $T_k \cup D_{k-},$ where $k\leq 1$.
\begin{corollary}\cite{AzizShah2}
		Let $r \in \mathcal{R}_n$ with all its zeros lie in $T_k \cup D_{k-},$ where $k \leq 1$. Then for $z \in T_1$,
				\begin{equation}
						|r'(z)| \geq \dfrac{1}{2} \left[ |B'(z)| + \dfrac{n(1-k)}{1+k} \right] \cdot |r(z)|. \label{eq1.8}
				\end{equation}
		Equality holds for $r(z) = \left(\dfrac{z+k}{z-a}\right)^n$ and \mbox{$B(z) = \left(\dfrac{1-az}{z-a}\right)^n$} evaluated at $z=1$, where $a > 1, k\leq 1.$
\end{corollary}

Next, we state our main results which generalize results by A. Aziz and W. M. Shah \cite{AzizShah, AzizShah2} and A. Aziz and B. A. Zarger \cite{AzizZar}. Their proofs will be presented in \mbox{section 3.}  The first theorem gives an estimate of an upper bound of the modulus of the derivative of $r(z)$ on the unit circle when all zeros of $r(z)$ lie in $|z|\geq k\geq 1.$

\begin{theorem}[Main] 	\label{thm1} 
		Let $r \in \mathcal{R}_n$, where $r$ has exactly $n$ poles at $a_1, a_2, \ldots, a_n$ and all its zeros lie in $T_k \cup D_{k+}$, where $k \geq 1.$  Then for $z\in T_1$,
\begin{equation}
|r'(z)| \leq \dfrac{1}{2}\left[|B'(z)| - \dfrac{(n(1+k)-2t)(|r(z)|-m)^2}{(1+k)(||r||-m)^2}\right](||r||-m), \label{eqthm7}
\end{equation}
		where $t$ is the number of zeros of $r$ with counting multiplicity and $\displaystyle m=\min_{|z|=k} |r(z)|.$ Equality holds for $r(z) = \dfrac{(z+k)^t}{(z-a)^n}$ and \mbox{$B(z) = \left(\dfrac{1-az}{z-a}\right)^n$} evaluated at $z=1$, where $a >1, k\geq 1.$
\end{theorem}


	From Theorem \ref{thm1}, if $r$ has all its zeros lie in $T_k \cup D_{k+}$ with at least one zero on $T_k$, then we obtain the following corollary.
\begin{corollary} 	\label{cor1.1} 
		Let $r \in \mathcal{R}_n$, where $r$ has exactly $n$ poles at $a_1, a_2, \ldots, a_n$ and all its zeros lie in $T_k \cup D_{k+}$ with at least one zero on $T_k$, where $k \geq 1.$  Then for $z\in T_1$,
\begin{equation}
|r'(z)| \leq \dfrac{1}{2}\left[|B'(z)| - \dfrac{n(1+k)-2t}{1+k}\cdot\dfrac{|r(z)|^2}{||r||^2}\right]\cdot||r||, \label{eqthm7}
\end{equation}
		where $t$ is the number of zeros of $r$ with counting multiplicity. Equality holds for $r(z) = \dfrac{(z+k)^t}{(z-a)^n}$ and \mbox{$B(z) = \left(\dfrac{1-az}{z-a}\right)^n$} evaluated at $z=1$, where $a >1, k\geq 1.$
\end{corollary}


As an immediate consequence of Theorem \ref{thm1}, we have the following generalization of Theorem \ref{thm-azizshah1} and Theorem \ref{thm-azizzar}, where $r$ has exactly $n$ zeros in $T_1 \cup D_{1+}$ and $r$ has exactly $n$ zeros in $T_k \cup D_{k+}, k\geq 1$, respectively.

\begin{corollary}\label{cor1} 
		Let $r(z)=p(z)/w(z) \in \mathcal{R}_n$ where $p(z)$ is a polynomial of degree $n$ and all its zeros lie in $T_k \cup D_{k+}$, $k \geq 1.$ 
		Then for $z \in T_1$,
				$$|r'(z)| \leq \dfrac{1}{2}\left[ |B'(z)| - \dfrac{n(k-1)(|r(z)|-m)^2}{(k+1)(||r||-m)^2}\right] (||r||-m),$$
		 where $\displaystyle m=\min_{|z|=k} |r(z)|.$ Equality holds for $r(z) = \left(\dfrac{z+k}{z-a}\right)^n$ and \mbox{$B(z) = \left(\dfrac{1-az}{z-a}\right)^n$} evaluated at $z=1$, where $a >1, k\geq 1.$
\end{corollary}

\indent  In particular, for $k=1$, Corollary \ref{cor1} reduces to Theorem \ref{thm-azizshah1}  and for $m=0$, Corollary \ref{cor1} reduces to Theorem \ref{thm-azizzar}.
 
The next theorem establishes an estimate of a lower bound of the modulus of the derivative of $r(z)$ on the unit circle when all zeros of $r$ lie in $|z| \leq  k \leq  1.$
\begin{theorem}[Main] \label{thm2}
			Let $r \in \mathcal{R}_n$, where $r$ has exactly $n$ poles at $a_1, a_2, \ldots, a_n$ and all its zeros lie in $T_k \cup D_{k-}$, where $k \leq 1.$				Then for $z\in T_1$, 
\begin{equation}
|r'(z)| \geq \dfrac{1}{2}\left[ |B'(z)| + \dfrac{2t-n(1+k)}{(1+k)}\right]\cdot(|r(z)|+m), \label{eqthm8}
\end{equation}						
			where $t$ is the number of zeros of $r$ with counting multiplicity and $\displaystyle m=\min_{|z|=k} |r(z)|.$ Equality holds for $r(z) = \dfrac{(z+k)^t}{(z-a)^n}$ and \mbox{$B(z) = \left(\dfrac{1-az}{z-a}\right)^n$} evaluated at $z=1$, where $a > 1, k\leq 1.$
\end{theorem}



As an immediate consequence of Theorem \ref{thm2}, we have the following generalization of Theorem \ref{thm-azizshah2}, where $r$ has exactly $n$ zeros in $T_1 \cup D_{1-}.$

\begin{corollary} \label{cor2} 
			Let $r(z)=p(z)/w(z) \in \mathcal{R}_n$ where $p(z)$ is a polynomial of degree $n$ and all its zeros lie in $T_k \cup D_{k-}$, $k \leq 1.$ 
			Then for $z \in T_1$,
					$$|r'(z)| \geq \dfrac{1}{2}\left[ |B'(z)| + \dfrac{n(1-k)}{(1+k)}\right]\cdot(|r(z)|+m),$$
			 where $\displaystyle m=\min_{|z|=k} |r(z)|.$ Equality holds for $r(z) = \left(\dfrac{z+k}{z-a}\right)^n$ and \mbox{$B(z) = \left(\dfrac{1-az}{z-a}\right)^n$} evaluated at $z=1$, where $a > 1, k\leq 1.$
\end{corollary}

\indent In particular, for $k=1$, Corollary \ref{cor2} reduces to Theorem \ref{thm-azizshah2}, and 
for $m=0$, Theorem \ref{thm2} reduces to Theorem \ref{thm-azizshahnew}. 

\section{Lemmas}
For the proof of our main theorems, we need the following lemmas. These three Lemmas are due to X. Li, R. N. Mohapatra and  R. S. Rodriguez \cite{Li}.
\begin{lem} \label{lem1}\cite{Li}
		If $r \in \mathcal{R}_n$ and $r^*(z) = B(z)\overline{r(1/\overline{z})}$, then for $z \in T_1$,
					$$|\left(r^*(z)\right)'| + |r'(z)| \leq |B'(z)|\cdot||r||.$$
		Equality holds for $r(z) = \lambda B(z)$ with $\lambda  \in T_1.$
\end{lem}

\begin{lem} \label{lem1.1}\cite{Li}
		Let $z\in \C$. Then
				$$\text{\emph{Re}}(z) \leq \dfrac{1}{2} \quad \text{if and only if} \quad |z|\leq |z-1|.$$
		Moreover, the statement holds when $\leq$ is replaced by $<$ at each occurrence.
\end{lem}
\begin{remark} \label{rem1.1} 
			Similar to the proof of Lemma \ref{lem1.1}, we can replace $\leq$ by $\geq$ and obtain that
					$$ \text{Re}\;(z) \geq \dfrac{1}{2} \quad \text{if and only if} \quad |z|\geq |z-1|.$$
\end{remark}

\begin{lem} \label{lem1.2} \cite{Li}
Let $r\in \mathcal{R}_n$. 
	\begin{enumerate}
			\item[(i)] 		If all zeros of $r$ lie in $T_1 \cup D_{1+}$, then for $z \in T_1$,
								$$\text{\emph{Re}}\left(\dfrac{zr'(z)}{r(z)}\right) \leq \dfrac{|B'(z)|}{2},$$  where $r(z) \neq 0$.
			\item[(ii)] 		If $r$ has exactly $n$ poles at $a_1,a_2,\ldots,a_n$ and all its zeros lie in $T_1 \cup D_{1-}$, then for $z \in T_1$,
								$$\text{\emph{Re}}\left(\dfrac{zr'(z)}{r(z)}\right) \geq \dfrac{|B'(z)|}{2} - \dfrac{1}{2}(n-t),$$
								where $t$ is the number of zeros of $r$ with counting multiplicity and $r(z) \neq 0$.
	\end{enumerate}
\end{lem}

The next lemma is due to A. Aziz and B. A. Zarger \cite{AzizZar}.
\begin{lem} \label{lem2}\cite{AzizZar}
		If $z \in T_1$, then $$\text{\emph{Re}} \left( \dfrac{zw'(z)}{w(z)} \right) = \dfrac{n-|B'(z)|}{2}.$$
\end{lem}
We need the following preliminary result for the proofs of Theorem \ref{thm1} and Theorem \ref{thm2}. 
\begin{lem} \label{lem3} 
		Assume that $r \in \mathcal{R}_n$, where $r$ has exactly $n$ poles at $a_1, a_2, \ldots, a_n$.
 Let $t$ be the number of zeros of $r$ with counting multiplicity.
	\begin{itemize}
			\item[$(i)$]		If  all zeros of $r$ lie in $T_k \cup D_{k+}$ where $k \geq 1$, and $z \in T_1$ with $r(z) \neq 0$, then 
											$$\text{\emph{Re}}\left(\dfrac{zr'(z)}{r(z)}\right) \leq \dfrac{|B'(z)|}{2} + \dfrac{2t-n(1+k)}{2(1+k)}.$$

			\item[$(ii)$]		If  all zeros of $r$ lie in $T_k \cup D_{k-}$ where $k \leq 1$, and $z \in T_1$ with $r(z) \neq 0$, then 
											$$\text{\emph{Re}}\left(\dfrac{zr'(z)}{r(z)}\right) \geq \dfrac{|B'(z)|}{2} - \dfrac{n(1+k)-2t}{2(1+k)}.$$
	\end{itemize}			
\end{lem}
\begin{proof} 
Let $r(z)=\dfrac{p(z)}{w(z)} \in \mathcal{R}_n.$ \\
If $b_1, b_2, \ldots, b_t$ are all zeros (may be not distinct) of $p(z)$, then $t \leq n$ 

	$(i)$ Assume that $|b_j| \geq k \geq 1,$ $j=1, 2, \ldots, t.$ Then	
			$$ \dfrac{zr'(z)}{r(z)} =\dfrac{zp'(z)}{p(z)}-\dfrac{zw'(z)}{w(z)}=\left[\sum_{j=1}^t \dfrac{z}{z-b_j}\right] - \dfrac{zw'(z)}{w(z)}.$$
For $z \in T_1$, this relation with the help of Lemma \ref{lem2} gives
		\begin{align}
					\text{Re} \left(\dfrac{zr'(z)}{r(z)}\right) &= \left( \sum_{j=1}^t \text{Re}\left(\dfrac{z}{z-b_j} \right)\right) - \left(\dfrac{n-|B'(z)|}{2}\right). \label{eqstar}
		\end{align}
For $z\in T_1$ with $z\neq b_j (1\leq j\leq t)$, we consider two cases.\\
\textbf{Case 1: \boldmath{$|b_j|=1$}.} Then $k=1$ and
						$\left\vert \dfrac{z}{z-b_j}\right\vert = \left\vert \dfrac{b_j}{z-b_j}\right\vert = \left\vert \dfrac{z}{z-b_j}-1\right\vert.$ \\
By Lemma \ref{lem1.1}, we get that
						$\text{Re} \left(\dfrac{z}{z-b_j}\right) \leq \dfrac{1}{2} = \dfrac{1}{1+1} = \dfrac{1}{1+k}.$\\						
\textbf{Case 2: \boldmath{$|b_j|>1$}.} A bilinear transformation $w_j(z) = \dfrac{z}{z-b_j}$ maps $T_1$ onto a circle\\ 
						$\left \{ w: \left\vert w+\dfrac{1}{|b_j|^2-1} \right\vert = \dfrac{|b_j|}{|b_j|^2-1 } \right \}.$
Then
		$$\text{Re} \left(\dfrac{z}{z-b_j}\right) 	\leq   \left(-\dfrac{1}{|b_j|^2-1}\right) + \dfrac{|b_j|}{|b_j|^2-1} = \dfrac{1}{|b_j|+1}  \leq \dfrac{1}{1+k}.$$
From both cases, $\text{Re} \left(\dfrac{z}{z-b_j}\right) \leq \dfrac{1}{1+k}$ for $|b_j| \geq k \geq 1$, $z\in T_1$ with \mbox{$z\neq b_j$ $(1\leq j \leq t)$.}\\ 
Substituting this relation into \eqref{eqstar}, we get that 
		$$ \text{Re} \left(\dfrac{zr'(z)}{r(z)}\right)  \leq \sum_{j=1}^t  \left(\dfrac{1}{1+k}\right) - \left(\dfrac{n-|B'(z)|}{2}\right) = \dfrac{|B'(z)|}{2} + \dfrac{2t-n(1+k)}{2(1+k)}. $$

	$(ii)$ Assume that $|b_j| \leq k \leq 1,$ $j=1, 2, \ldots, t.$ Then		
					$$\dfrac{zr'(z)}{r(z)} =\dfrac{zp'(z)}{p(z)}-\dfrac{zw'(z)}{w(z)}=\left[\sum_{j=1}^t \dfrac{z}{z-b_j}\right] - \dfrac{zw'(z)}{w(z)}.$$
For $z \in T_1$, this relation with the help of Lemma \ref{lem2} gives
		\begin{align}
				\text{Re} \left(\dfrac{zr'(z)}{r(z)}\right) &= \left( \sum_{j=1}^t \text{Re}\left(\dfrac{z}{z-b_j} \right)\right) - \left(\dfrac{n-|B'(z)|}{2}\right). \label{eqstar2}
		\end{align}
For $z\in T_1$ with $z\neq b_j (1\leq j\leq t)$, we consider two cases.\\
\textbf{Case 1: \boldmath{$|b_j|=1$}.} Then $k=1$ and
				$\left\vert \dfrac{z}{z-b_j}\right\vert = \left\vert \dfrac{b_j}{z-b_j}\right\vert = \left\vert \dfrac{z}{z-b_j}-1\right\vert.$\\
By Remark \ref{rem1.1}, we get that
				$\text{Re} \left(\dfrac{z}{z-b_j}\right) \geq \dfrac{1}{2} = \dfrac{1}{1+1} = \dfrac{1}{1+k}.$\\
\textbf{Case 2: \boldmath{$|b_j|<1$}.} A bilinear transformation $w_j(z) = \dfrac{z}{z-b_j}$ maps $T_1$ onto a circle\\ 
	$\left \{\left\vert w-\left(\dfrac{1}{1-|b_j|^2}\right) \right\vert = \dfrac{|b_j|}{1-|b_j|^2} \right \}.$\\
Then
		$ \text{Re} \left(\dfrac{z}{z-b_j}\right) 		\geq 		\left(\dfrac{1}{1-|b_j|^2}\right) - \dfrac{|b_j|}{1-|b_j|^2} 
																		= 			\dfrac{1}{1+|b_j|}
																		\geq 		\dfrac{1}{1+k}.$\\
From both cases, $\text{Re} \left(\dfrac{z}{z-b_j}\right) \geq \dfrac{1}{1+k}$ for $|b_j| \leq k \leq 1$, $z\in T_1$ with  \mbox{$z\neq b_j$ $(1\leq j \leq t)$.}\\
Substituting this relation into \eqref{eqstar2}, we get that
	$$\text{Re} \left(\dfrac{zr'(z)}{r(z)}\right) 	\geq  	\sum_{j=1}^t  \left(\dfrac{1}{1+k}\right) - \left(\dfrac{n-|B'(z)|}{2}\right)
																	=			\dfrac{|B'(z)|}{2} - \dfrac{n(1+k)-2t}{2(1+k)}.$$
\end{proof}
\vskip0.5cm

\section{Proofs of the main theorems}
\indent In this section, we present the proofs of our main results. \\ \noindent
\emph{Proof of Theorem \ref{thm1}.} 
		Assume that $r \in \mathcal{R}_n$ has no zeros in $|z| < k$, where $k \geq 1$.\\
Let $\displaystyle m=\min_{|z|=k} |r(z)|$ and $t$ be the number of zeros of $r$ with counting multiplicity. \\ 
If $r(z)$ has a zero on $|z|=k$, then $m=0$ and hence for every $\alpha$ with $|\alpha| <1$, we get $r(z)-\alpha m = r(z).$
In case $r(z)$ has no zeros on $|z|=k$, we have for every $\alpha$ with $|\alpha| <1$ that 
			$$|-\alpha m| = |\alpha|\cdot m < |r(z)|\quad\text{for}\; |z|=k.$$
It follows from Rouche's theorem that $R(z) = r(z) -\alpha m $ and $r(z)$ have the same number of zeros in $\{ |z|<k\}.$ 
That is, for every $\alpha$ with $|\alpha|<1$, $R(z)$ has no zeros in $|z| < k$.
We assume that $R(z) \neq 0.$ Lemma \ref{lem3} (i) yields that for $z\in T_1$,
		\begin{equation}
					\text{Re}\;\left(\dfrac{zR'(z)}{R(z)}\right) \leq \dfrac{|B'(z)|}{2} + \dfrac{2t-n(1+k)}{2(1+k)}. 	 \label{heart}
		\end{equation}
Note that $R^*(z) = B(z)\overline{R(1/\overline{z})} = B(z)\overline{R}(1/z).$ Then
		\begin{align*}
				(R^*(z))' 	&=  	\overline{R}(1/z)B'(z)+B(z)\left( \overline{R}(1/z) \right)' \\
								&=	B'(z)\overline{R}(1/z)+B(z)\left( \overline{R}'(1/z) \right)\left( -\dfrac{1}{z^2} \right) \\
								&=	B'(z)\overline{R}(1/z)-\dfrac{B(z)}{z^2}\cdot \overline{R}'(1/z).
		\end{align*} 
Then \quad $z(R^*(z))'=zB'(z)\overline{R}(1/z)-\dfrac{B(z)}{z}\cdot \overline{R}'(1/z).$
Since $z\in T_1$, we have $\overline{z}=\dfrac{1}{z},$\\ \mbox{$|B(z)|=1$,} $\dfrac{zB'(z)}{B(z)} = |B'(z)|$, and so
		\begin{align*}
				|z(R^*(z))'| 		&= \left\vert zB'(z)\overline{R(z)}-B(z)\overline{zR'(z)} \right\vert\\
										&= \left\vert \dfrac{zB'(z)}{B(z)}\cdot\overline{R(z)}-\overline{zR'(z)}\right\vert \\
										&= \left\vert |B'(z)|\overline{R(z)}-\overline{zR'(z)}\right\vert.
		\end{align*}
Since $|B'(z)|$ is real, we get \quad $|z(R^*(z))'| = \left\vert |B'(z)|R(z) - zR'(z)\right\vert.$ \\
Then 
		\begin{align*}
					\left\vert \dfrac{z(R^*(z))'}{R(z)} \right\vert ^2 
								&= 		\left\vert |B'(z)| - \dfrac{zR'(z)}{R(z)} \right\vert ^2\\
								&= 		|B'(z)|^2 -2|B'(z)|\cdot\text{Re} \left(\dfrac{zR'(z)}{R(z)}\right) + \left\vert \dfrac{zR'(z)}{R(z)} \right\vert^2 \\
								&\geq 	|B'(z)|^2 -2|B'(z)|\left[ \dfrac{|B'(z)|}{2} + \dfrac{2t-n(1+k)}{2(1+k)} \right] + \left\vert \dfrac{zR'(z)}{R(z)} \right\vert^2\\
								&= 		\left\vert \dfrac{zR'(z)}{R(z)} \right\vert^2 + \dfrac{n(1+k)-2t}{(1+k)}\cdot|B'(z)|,
		\end{align*}
where the inequality comes from \eqref{heart}.\\
This implies that for $z \in T_1$, 
		\begin{align}
					\left[ |R'(z)|^2 + \dfrac{n(1+k)-2t}{(1+k)}\cdot|R(z)|^2|B'(z)|\right]^{\frac{1}{2}} \leq |(R^*(z))'|, \label{eq1}
		\end{align}
where $R^*(z) = B(z) \overline{R(1/\overline{z})} = r^*(z) - \overline{\alpha} m B(z).$\\
Moreover,  $(R^*(z))' = (r^*(z))' - \overline{\alpha} m B'(z) $ and $R'(z) = (r(z)-\alpha m)' = r'(z) $.\\
Apply these relations into \eqref{eq1}, we obtain that
		\begin{align}
					\left[ |r'(z)|^2 + \dfrac{n(1+k)-2t}{(1+k)}\cdot|r(z)-\alpha m|^2|B'(z)|\right]^{\frac{1}{2}} \leq |(r^*(z))' -  \overline{\alpha} m B'(z)|, \label{eq2}
		\end{align}
for $z \in T_1$ and for $\alpha$ with $|\alpha| <1$.\\
\noindent Choose the argument of $\alpha$ such that
		\begin{align}
				|(r^*(z))' -  \overline{\alpha} m B'(z)| = |(r^*(z))'| - m\left\vert\alpha\right\vert |B'(z)|,  	\label{eq3}
		\end{align}
for $z \in T_1$.\\
Triangle inequality yields that $|r(z)-m\alpha | \geq \left\vert |r(z)| - m |\alpha|\right\vert. $
Note that $||r(z)|-m|\alpha| |^2 = (|r(z)|-m|\alpha|)^2$ which implies that
		\begin{align}
					|r(z)-m\alpha |^2 \geq  (|r(z)| - m |\alpha|)^2. \label{eq4}
		\end{align}
Substituting relations \eqref{eq3} and \eqref{eq4} into \eqref{eq2}, we obtain that
			$$\left[ |r'(z)|^2 + \dfrac{n(1+k)-2t}{(1+k)}\cdot(|r(z)|-m|\alpha|)^2|B'(z)|\right]^{\frac{1}{2}} \leq |(r^*(z))'| - m\left\vert\alpha\right\vert |B'(z)|.$$
Letting $|\alpha| \rightarrow 1$, we get
			$$\left[ |r'(z)|^2 + \dfrac{n(1+k)-2t}{(1+k)}\cdot(|r(z)|-m)^2|B'(z)|\right]^{\frac{1}{2}} \leq |(r^*(z))'| - m|B'(z)|.$$
Lemma \ref{lem1} implies that
		$$\left[ |r'(z)|^2 + \dfrac{n(1+k)-2t}{(1+k)}\cdot(|r(z)|-m)^2|B'(z)|\right]^{\frac{1}{2}} \leq |B'(z)|\cdot||r|| - |r'(z)| - m|B'(z)|.$$
Equivalently,
		$$ |r'(z)|^2 + \dfrac{n(1+k)-2t}{(1+k)}\cdot(|r(z)|-m)^2|B'(z)| \leq \bigg[ (||r||-m)|B'(z)| - |r'(z)| \bigg]^2.$$
Hence,
		\begin{align*}
				|r'(z)|^2 + \dfrac{n(1+k)-2t}{(1+k)}	&			\cdot(|r(z)|-m)^2|B'(z)| \\
																		&\leq 	(||r||-m)^2|B'(z)|^2 -2(||r||-m)|B'(z)||r'(z)|+|r'(z)|^2.
		\end{align*}
Then
		\begin{align*}
					2(||r||-m)|r'(z)| \leq (||r||-m)^2|B'(z)| - \dfrac{n(1+k)-2t}{(1+k)}(|r(z)|-m)^2.
		\end{align*}
Thus,
		\begin{align*}
				|r'(z)| 		&\leq \dfrac{(||r||-m)^2|B'(z)|}{2(||r||-m)} - \dfrac{(n(1+k)-2t)(|r(z)|-m)^2}{2(1+k)(||r||-m)}\\
								&= \dfrac{1}{2}\left[|B'(z)| - \dfrac{(n(1+k)-2t)(|r(z)|-m)^2}{(1+k)(||r||-m)^2}\right](||r||-m),
		\end{align*}
where $t$ is the number of zeros of $r$ with counting multiplicity and $\displaystyle m=\min_{|z|=k} |r(z)|.$\\
This proves inequality for $R(z) \neq 0.$ In case $R(z)=0$, we obtain that $r'(z) =0.$ This implies that the above inequality is trivially  true.\\ \indent
Therefore, inequality holds for all $z \in T_1.$

To show that equality \eqref{eqthm7} holds for $r(z) = \dfrac{(z+k)^t}{(z-a)^n}$ and \mbox{$B(z) = \left(\dfrac{1-az}{z-a}\right)^n,$} where \mbox{$a >1,$} $k\geq 1$ at $z=1$, we observe that
		$$||r||=\sup_{z\in T_1} |r(z)| = \dfrac{(1+k)^t}{(1-a)^n} = |r(1)|, \;m = \min_{|z|=k} |r(z)| = 0, \;\textrm{and} \;\;|B'(1)| = \dfrac{n(a+1)}{a-1}.$$
Since		$r'(z) 	= 		\left[\dfrac{t}{z+k} + \dfrac{n}{a-z}\right] \left[\dfrac{(z+k)^t}{(z-a)^n}\right], $ we obtain that
			$$|r'(1)| 	= 		\left[\dfrac{t}{1+k} + \dfrac{n}{a-1}\right] \left[\dfrac{(1+k)^t}{(1-a)^n}\right] = \left[\dfrac{t}{1+k} + \dfrac{n}{a-1}\right]||r||.$$
The right side of the relation \eqref{eqthm7} is
		\begin{align*}
			\dfrac{1}{2}\left[|B'(1)| - \dfrac{(n(1+k)-2t)(|r(1)|-m)^2}{(1+k)(||r||-m)^2}\right](||r||-m) 	
						&= 		\dfrac{1}{2}\left[\dfrac{n(a+1)}{a-1} - n +\dfrac{2t}{1+k}\right]||r||\\
						&= 		\dfrac{1}{2}\left[\dfrac{2n}{a-1} +\dfrac{2t}{1+k}\right]||r||\\
						&= 		|r'(1)|.
		\end{align*}
This proves Theorem \ref{thm1} completely. \\
\begin{remark} We show that our upper bound in Theorem \ref{thm1} improves an upper bound in Theorem \ref{thm-azizshah1} as follows.\\
Since $t \leq n$ and $k\geq 1$, we get that
			\begin{align*}
				\dfrac{(n(1+k)-2t)(|r(z)|-m)^2}{(1+k)(||r||-m)^2} &\geq \dfrac{(n(1+k)-2n)(|r(z)|-m)^2}{(1+k)(||r||-m)^2}\\
											&\geq \dfrac{(n(2)-2n)(|r(z)|-m)^2}{(1+k)(||r||-m)^2}\\
											&=0.
			\end{align*}
							Hence, $|B'(z)|- \dfrac{(n(1+k)-2t)(|r(z)|-m)^2}{(1+k)(||r||-m)^2}\leq |B'(z)|.$ \\
							In particular, if $t=n$ and $k=1$, then $|B'(z)|- \dfrac{(n(1+k)-2t)(|r(z)|-m)^2}{(1+k)(||r||-m)^2}= |B'(z)|.$\\
							Therefore, our upper bound in Theorem \ref{thm1} is better than an upper bound in Theorem \ref{thm-azizshah1}.
\end{remark}
\indent

Next, we give the proof of the second main result.\\

\emph{Proof of Theorem \ref{thm2}.} 
Assume that $r \in \mathcal{R}_n$ has no zeros in $|z| > k$, where $k \leq 1$.\\
Let $\displaystyle m=\min_{|z|=k} |r(z)|$ and $t$ be the number of zeros of $r$ with counting multiplicity. \\
Then $m \leq |r(z)|$ for $z \in T_k$.\\
If $r(z)$ has a zero on $|z|=k$, then $m=0$ and hence for every $\alpha$ with $|\alpha| <1$, we get $r(z)+\alpha m = r(z)$. 
In case $r(z)$ has no zeros on $|z|=k$, we have for every $\alpha$ with $|\alpha| <1$ that  $|\alpha m| < |r(z)|\quad\text{for}\; |z|=k.$\\
It follows from Rouche's theorem that $R(z) = r(z) +\alpha m $ and $r(z)$ have the same number of zeros in $\{ |z|<k\}.$ 
That is, for every $\alpha$ with $|\alpha|<1$, $R(z)$ has no zeros in $|z| > k$. 
We assume that $R(z) \neq 0.$  Lemma \ref{lem3} (ii) implies that for $z\in T_1$,
		$$\text{Re}\;\left(\dfrac{zR'(z)}{R(z)}\right) \geq \dfrac{|B'(z)|}{2} - \dfrac{n(1+k)-2t}{2(1+k)},$$
where $R(z) \neq 0$.\\
Then
		$$\left\vert\dfrac{R'(z)}{R(z)}\right\vert=\left\vert\dfrac{zR'(z)}{R(z)}\right\vert \geq \text{Re}\;\left(\dfrac{zR'(z)}{R(z)}\right) \geq \dfrac{|B'(z)|}{2} - \dfrac{n(1+k)-2t}{2(1+k)}.$$
This implies that 
		$$|R'(z)|\geq \left[\dfrac{|B'(z)|}{2} - \dfrac{n(1+k)-2t}{2(1+k)}\right]\cdot|R(z)|,    \qquad\qquad\text{for}\;z\in T_1.$$
Since $|R'(z)| = |r'(z)|$, we obtain that
		$$|r'(z)| \geq \dfrac{1}{2}\left[ |B'(z)| - \dfrac{n(1+k)-2t}{(1+k)}\right]\cdot|r(z)+\alpha m|,   \qquad\text{for}\;z\in T_1.$$
Note that this inequality is trivially true for $R(z)=0.$ Therefore, this inequality holds for all $z \in T_1.$\\
Choosing the argument of $\alpha$ suitably in the right side of the above inequality and noting that the left side is independent of $\alpha$, we get that
		$$|r'(z)| \geq \dfrac{1}{2}\left[ |B'(z)| - \dfrac{n(1+k)-2t}{(1+k)}\right]\cdot(|r(z)|+|\alpha| m),   \qquad\text{for}\;z\in T_1.$$
Letting $|\alpha| \rightarrow 1$, we get  for $z\in T_1$ that
\begin{align*}
|r'(z)| &\geq \dfrac{1}{2}\left[ |B'(z)| - \dfrac{n(1+k)-2t}{(1+k)}\right]\cdot(|r(z)|+m)\\
&=\dfrac{1}{2}\left[ |B'(z)| + \dfrac{2t-n(1+k)}{(1+k)}\right]\cdot(|r(z)|+m),
\end{align*}
where $t$ is the number of zeros of $r$ with counting multiplicity and $\displaystyle m=\min_{|z|=k} |r(z)|$.\\ \indent
As well as the proof of Theorem \ref{thm1}, we can show that equality \eqref{eqthm8} holds for $r(z) = \dfrac{(z+k)^t}{(z-a)^n}$ and \mbox{$B(z) = \left(\dfrac{1-az}{z-a}\right)^n,$} where \mbox{$a >1,$} $k\leq 1$ at $z=1$. 

\section{Conclusion}
This paper investigates the estimate of the modulus of the derivative of $r(z)$ on the unit circle. We establish an upper bound when all zeros of $r(z)$ lie in $|z|\geq k\geq 1$ and a lower bound when all zeros of $r(z)$ lie in $|z|\leq  k \leq 1.$ In particular, if $r(z)$ has exactly $n$ zeros and $k=1$, our main theorems generalize results by  A. Aziz and W. M. Shah \cite{AzizShah} and B. A. Zarger \cite{AzizZar}. Furthermore, if $r(z)$ has a zero on $T_k$, the second main result generalizes a result by A. Aziz and W. M. Shah \cite{AzizShah2}.\\


\textit{Acknowledgements.}
	This work has received a scholarship under the Research Fund for Supporting Lecturer to Admit High Potential Student to Study and Research on His Expert Program Year 2018 from the Graduate School, Khon Kaen University, Thailand (Grant no. 612JT217).



\end{document}